\documentclass[11pt, reqno]{amsart}
\usepackage{amscd}        

\usepackage{float}

\usepackage{nomencl}

\usepackage{algorithm}
\usepackage{hyperref}

\usepackage{graphicx}
\usepackage{rotating}
\usepackage{amssymb}
\usepackage{epstopdf}
\usepackage{tikz}
\definecolor{mintgreen}{RGB}{152,255,152}
\definecolor{pinksalmon}{RGB}{255,102,102}
\definecolor{hueso}{RGB}{245,245,220}
\definecolor{marfil}{RGB}{255,253,208}
\definecolor{amarillo}{RGB}{255,255,0}
\definecolor{mygreen}{RGB}{0,186,123}
\usetikzlibrary{decorations.markings,arrows}
\usetikzlibrary{decorations.pathreplacing}
\usepackage[inner=1.0in,outer=1.0in,bottom=1.0in, top=1.0in]{geometry}

\numberwithin{equation}{section}

\newtheorem{theorem}{Theorem}[section]
\newtheorem{lemma}[theorem]{Lemma}
\newtheorem{proposition}[theorem]{Proposition}
\newtheorem{corollary}[theorem]{Corollary}
\newtheorem{conjecture}[theorem]{Conjecture}

\makeatletter
\def\moverlay{\mathpalette\mov@rlay}
\def\mov@rlay#1#2{\leavevmode\vtop{%
   \baselineskip\z@skip \lineskiplimit-\maxdimen
   \ialign{\hfil$\m@th#1##$\hfil\cr#2\crcr}}}
\newcommand{\charfusion}[3][\mathord]{
    #1{\ifx#1\mathop\vphantom{#2}\fi
        \mathpalette\mov@rlay{#2\cr#3}
      }
    \ifx#1\mathop\expandafter\displaylimits\fi}
\makeatother

\newcommand{\suchthat}{\;\ifnum\currentgrouptype=16 \middle\fi|\;}

\newcommand{\Z}{\mathbb{Z}}
\newcommand{\C}{\mathbb{C}}
\newcommand{\Q}{\mathbb{Q}}
\newcommand{\R}{\mathbb{R}}

\newcommand{\op}[1]{\operatorname{#1}}

\newcommand{\Tr}[1]{\operatorname{Tr}#1}

\def\disc{{\textrm{disc}}}
\newcommand{\kro}[2]{\left( \tfrac{#1}{#2}\right )}
\newcommand{\tr}{\operatorname{tr}}
\newcommand{\adj}{\operatorname{adj}}

\theoremstyle{definition}
\newtheorem{remark}[theorem]{Remark}
\newtheorem{definition}[theorem]{Definition}

\newtheorem*{question}{Question}

\def \F { \mathbb{F}}
\def\Z{{\mathbb Z}}
\def\Q{{\mathbb Q}}
\def\H{{\mathbb H}}
\def \gal { \text{Gal}}

\begin{document}

\title{Theta series and number fields:  theorems and experiments}
\author{}

\author{Adrian Barquero-Sanchez, Guillermo Mantilla-Soler and Nathan C. Ryan}





\begin{abstract}
  We construct certain $\theta$-series associated to number fields and prove that for number fields of degree less than equal to 4, these $\theta$-series are number field invariants.  We also investigate whether or not the collection of $\theta$-series associated to number fields of the same degree and discriminant are linearly independent.  This is known to be true if the degree of the number field is less than or equal to 3.  We do not prove in this paper that they are linearly independent in general but we do give computational and heuristic evidence that we would expect them to be.
\end{abstract}

\maketitle

\section{Introduction and statement of results} 

Quadratic forms and number fields have been intertwined since the beginning of the development of algebraic number theory.  An important method to create modular forms is the construction of $\theta$-series associated to quadratic forms.  In this paper, we investigate the relationship between modular forms and number fields.  In \cite{zagier123}, a particular modular form, a linear combination of two $\theta$-series is shown to reveal a great deal of information about the quadratic forms of discriminant $-23$, the quadratic field $\Q(\sqrt{-23})$ and modular forms of level 23.  Perhaps the most extensively studied modular form is the Jacobi's $\theta$-function:  
\[
\vartheta(z) = \sum_{n=-\infty}^\infty q^{n^2}
\]
where $q=e^{\pi i z}$ for $z$ in the complex upper half-plane.  Let $d$ be a square free positive integer. A slight generalization of  Jacobi's form is the modular form $\displaystyle \vartheta_{d}(z) = \sum_{n=-\infty}^\infty q^{dn^2}$. Using the level, or by more elementary arguments, we see that the function $\displaystyle \vartheta_{d}(z)$ determines the value $d$. This could be interpreted in terms of number fields as giving a complete invariant for totally real number fields:  Let $K$ be a real quadratic number field and let $d_{K}$ be the unique square free positive integer such that $K =\Q(\sqrt{d_{K}})$. In this context the assignment \[K \mapsto \vartheta_{d_{K}}\] is a complete invariant. It is an invariant for $K$ since the map is independent of the isomorphism class of $K$, and it is complete since it is injective. This simple construction leads one naturally to consider the following:

\begin{question}
In what ways can one associate a modular form to a number field so that one can recover meaningful information about the number field?
\end{question}

More broadly, one of the main problems in algebraic number theory is to give a satisfactory way of deciding whether or not two number fields are isomorphic:
\begin{question}\label{LaPregunta}
Can one describe an isomorphism between number fields $K$, $L$ from an associated mathematical/arithmetical object? In other words, can we find a complete invariant for number fields?
\end{question}

One of the purposes of this paper is to generalize the construction $K \mapsto \vartheta_{d_{K}}$ from degree $n=2$ to arbitrary higher degree $n$, and to show, that  for $n$ up to $4$,  such generalization is a complete invariant.

Several natural objects have been at the center of study of the second question by several authors; e.g., the Dedekind zeta function, the ring of adeles, the group of Dirichlet characters, the absolute Galois group (see \cite{Corn1, Iwa, Komat, Neu, Perlis1, Uchida, Uchida1}). Among all of these, the only one that is a complete invariant, due to a famous result of Neukirch and Uchida, see \cite{Neu, Uchida, Uchida1}, is the absolute Galois group.  
All the  invariants mentioned above, ours included, are all refinements of the discriminant; i.e., having the same invariant for $K$ and $L$ implies that the two fields have the same discriminant. Since the discriminant of a number field $K$ is the determinant of the trace pairing, with respect any integral basis, it is natural to consider the isometry class of the integral bilinear pairing \begin{displaymath}
\begin{array}{cccc}
\tr_{K/\Q}: & O_{K} \times O_{K} &\rightarrow& \Z  \\  & (x,y) & \mapsto &
\tr_{K/\Q}(xy).
\end{array}
\end{displaymath}
 as the first natural object refining the discriminant. 
 
In \cite{GMS1} the second author associated a space of weight 1 modular forms to totally real cubic number fields.  In particular, he showed that there is an injection from the set of isomorphism classes of totally real cubic number fields of fundamental discriminant $d$ to the space of weight one modular forms of a prescribed level and character determined by $d$.  Moreover, he showed that the forms in the image of this map are linearly independent.  The modular forms that he creates are $\theta$-series derived from the map
 \begin{align*}
 t_K^0 : & \ O_K^0 \to \Z\\
 &x\mapsto \tfrac12\tr_{K/\Q}(x^2)
\end{align*}
where $K$ is a number field and $O_K^0\subset K$ is the set of integral elements with trace zero.  This is shown to be a quadratic form over $\Z$ and the associated modular form he constructs is the theta series of the quadratic form.  Combining other results of the second author \cite{GMS2, GMS3} one concludes that these modular forms are invariants for number fields in the sense described above. In this paper, we study the same question but in greater generality.  First for a degree $n$ totally real number field $K$, of fundamental discriminant co-prime to $n$, we define a theta series $\theta_{K}$ that generalizes the definition of  \cite{GMS1} to all $n$.

\begin{theorem}\label{LaDefi}
Let $K$ be a totally real number field of degree $n$ with discriminant $d_K$. Then if $d_K$ is a fundamental discriminant and $\gcd{(n, d_K)} = 1$, the theta series
\begin{align*}
\theta_K(z) := \sum_{\alpha \in \mathcal{O}_K^{0}} e^{\pi i \Tr_{K/\Q}{(\alpha^2)} z} \qquad (z \in \H),
\end{align*}
satisfies that
\begin{align*}
\theta_K \in M_{\frac{n - 1}{2}} \left( \Gamma_0(2nd_K), \left( \frac{\delta_{n} nd_{K}}{\cdot} \right) \right),
\end{align*}
where $\delta_n =\begin{cases}
 
(-1)^{\frac{n-1}{2}} &\text{if $n$ is odd,}\\
\frac{1}{2} &\text{if $n$ is even.}
\end{cases}$ 
\end{theorem}

\begin{remark}
By convention we define $\theta_{\Q}(z)$ to be the constant function $1$.
\end{remark}

Generalizing the results of  \cite{GMS1} from degree $3$ to degree $4$ we show that there is an injection from the set of isomorphic classes of totally real quartic number fields of square free discriminant $d$ to a prescribed space of weight $3/2$ modular forms.  

\begin{theorem}\label{MainTheorem}
Let $K$ and $L$ be totally real number fields of degree $n \leq 4$ with discriminants $d_K = d_L = d$, where $d$ is a fundamental discriminant. Then if $\gcd{(n, d)} = 1$, we have
\begin{align*}
K \cong L \iff \theta_K = \theta_L,
\end{align*}
i.e., the theta series $\theta_K$ completely determines the isomorphism class of the number field $K$.
\end{theorem}

 At the moment  we cannot show, as in the cases $n \leq 3$, that $\theta$-series in the image are linearly independent but we give both computational evidence and heuristic evidence (see \S \ref{Heuristics} for why we expect them to be linearly independent).\\

\noindent
This paper is organized as follows.  First, we provide necessary background and prove some basic theorems about the objects we are studying.  Second, we prove the theorems stated above and we conclude with a summary of various computations and heuristics related to our expectation that the modular forms in the image of the map we define will be linearly independent.

\section*{Acknowledgments}

We would like to thank the organizers of the workshop Number Theory in the Americas for providing us the opportunity to start this collaboration and the Casa Matem\'atica Oaxaca (CMO-BIRS) for the hospitality.

\section{Some background on quadratic forms}\label{QuadraticBackground}

In this section we gather some basic background on quadratic forms and quadratic spaces over $\Z$ that will be needed later in the paper. We start by recalling the definition of a quadratic form on a free $\Z$-module.

\begin{definition}
A quadratic form on a free $\Z$-module $\Lambda$ is a function $\Phi: \Lambda \longrightarrow \Z$ such that
\begin{itemize}
\item[(i)] $\Phi(m \lambda) = m^2 \Phi(\lambda)$ for every $m \in \Z$ and every $\lambda \in \Lambda$;
\item[(ii)] The function $\phi: \Lambda \times \Lambda \longrightarrow \Z$ defined by 
\begin{align*}
\phi(\alpha, \beta) := \Phi(\alpha + \beta) - \Phi(\alpha) - \Phi(\beta)
\end{align*}
is a symmetric bilinear form on $\Lambda$.
\end{itemize}
The functions $\Phi$ and $\phi$ are said to be associated.
\end{definition}

\begin{definition}
An ordered pair $(\Lambda, \Phi)$ with $\Lambda$ a free $\Z$-module of finite rank and $\Phi: \Lambda \longrightarrow \Z$ a quadratic form on $\Lambda$ will be called a quadratic space on $\Z$.
\end{definition}

Now, if the free $\Z$-module $\Lambda$ has finite rank $n$ and $\mathcal{B} = \{ \alpha_1, \dots, \alpha_n  \}$ is a $\Z$-basis for $\Lambda$, there is an associated symmetric matrix $A_{\Phi, \mathcal{B}} \in M_{n\times n}(\Z)$ whose entries are given by 
\begin{align*}
(A_{\Phi, \mathcal{B}})_{i, j} := \phi(\alpha_i, \alpha_j),
\end{align*}
for $1 \leq i, j \leq n$. Clearly the matrix $A_{\Phi,\mathcal{B}}$ depends on the choice of $\Z$-basis for $\Lambda$. If we have two $\Z$-bases $\mathcal{B}$ and $\mathcal{C}$ for $\Lambda$ and $P = P_{\mathcal{B}}^{\mathcal{C}} \in \op{GL}_n(\Z)$ is the change of basis matrix from $\mathcal{C}$ to $\mathcal{B}$, then 
\begin{align*}
A_{\Phi, \mathcal{C}} = P^t A_{\Phi, \mathcal{B}} P.
\end{align*}
The rank of the quadratic form $\Phi$ is the rank of the matrix $A_{\Phi, \mathcal{B}}$ with respect to any $\Z$-basis $\mathcal{B}$. Moreover, we define the discriminant of $\Phi$ to be $\Delta_{\Phi}= \op{disc}(\Phi) := (-1)^{\frac{n(n-1)}{2}} \det{(A_{\Phi, \mathcal{B}})}$.

Observe that the matrix $A_{\Phi, \mathcal{B}}$ is what is called an even matrix, i.e., a matrix in $M_{n \times n} (\Z)$ such that its diagonal entries lie in $2 \Z$. This is because for any $x \in \Lambda$ we have $\phi(x, x) = \Phi(2x) - 2 \Phi(x) = 2\Phi(x)$ and hence the diagonal entries of $A_{\Phi, \mathcal{B}}$ are
\begin{align*}
(A_{\Phi, \mathcal{B}})_{i, i} := \phi(\alpha_i, \alpha_i) = 2 \Phi(\alpha_i) \in 2 \Z.
\end{align*}

Now, if $[x]_{\mathcal{B}} \in \Z^n$ denotes the vector of coordinates of an element $x \in \Lambda$ in the basis $\mathcal{B}$, then 
\begin{align*}
\phi(x, y) = [x]_{\mathcal{B}}^{t} A_{\Phi, \mathcal{B}} [y]_{\mathcal{B}}.
\end{align*}
Therefore, recalling that $\phi(x, x) = 2\Phi(x)$, we have
\begin{align*}
\Phi(x) = \frac{1}{2}[x]_{\mathcal{B}}^{t} A_{\Phi, \mathcal{B}} [x]_{\mathcal{B}}.
\end{align*}
Thus, if $X = \begin{bmatrix} x_1 \\ \vdots \\ x_n\end{bmatrix}$ we can associate a homogeneous quadratic polynomial 
\begin{align*}
Q_{\Phi, \mathcal{B}}(x_1, \cdots, x_n) := \frac{1}{2} X^t A_{\Phi, \mathcal{B}}  X = \frac{1}{2} [x_1, \dots, x_n] A_{\Phi, \mathcal{B}} \begin{bmatrix} x_1 \\ \vdots \\ x_n\end{bmatrix} = \sum_{i = 1}^n \frac{\phi(\alpha_i, \alpha_i)}{2} x_{i}^2 + \sum_{i < j} \phi(\alpha_i, \alpha_j) x_i x_j.
\end{align*}
Moreover, in terms of this polynomial we have that if $[x]_{\mathcal{B}} = [x_1, \dots, x_n]^t$, then
\begin{align}\label{Phi_Quadratic}
\Phi(x) = Q_{\Phi, \mathcal{B}}(x_1, \dots, x_n).
\end{align}
The $\phi(\alpha_i,\alpha_j)$ are called the coefficients of the quadratic form.

A quadratic form $\Phi$ is said to be positive definite if the matrix $A_{\Phi, \mathcal{B}}$, thought as a matrix in $M_{n \times n}(\R)$, is positive definite.

\begin{definition}
If $\Phi_1: \Lambda_1 \longrightarrow \Z$ and $\Phi_2: \Lambda_2 \longrightarrow \Z$ are quadratic forms on $\Z$-modules $\Lambda_1$ and $\Lambda_2$, they are said to be isometric if there is a $\Z$-isomorphism $f: \Lambda_1 \longrightarrow \Lambda_2$ such that $\Phi_2(f(x)) = \Phi_1(x)$ for every $x \in \Lambda_1$. Moreover, the corresponding quadratic spaces $(\Lambda_1, \Phi_1)$ and $(\Lambda_2, \Phi_2)$ are said to be isomorphic if $\Phi_1$ and $\Phi_2$ are isometric and we write $(\Lambda_1, \Phi_1) \cong (\Lambda_2, \Phi_2)$ in this case.
\end{definition}

\section{Quadratic forms and theta series}

Let $\Lambda$ be a free $\Z$-module of rank $n$ and let $\Phi: \Lambda \longrightarrow \Z$ be a positive definite quadratic form on $\Lambda$. Then the theta series associated to $\Phi$ is the function
\begin{align*}
\theta_{\Phi}(z) := \sum_{\lambda \in \Lambda} q^{\Phi(\lambda)} = \sum_{\lambda \in \Lambda} e^{2 \pi i \Phi(\lambda) z},
\end{align*}
where as usual $q = e^{2 \pi i z}$ for $z \in \H$. Requiring $\Phi$ to be positive definite guarantees that the series defining $\theta_{\Phi}$ converges absolutely in $\H$ and moreover $\theta_{\Phi}$ is a holomorphic function on $\H$ (see e.g. \cite[chapter VI]{Ogg69}). Now, choosing a $\Z$-basis $\mathcal{B}$ for $\Lambda$, equation (\ref{Phi_Quadratic}) allows us to write
\begin{align*}
\theta_{\Phi}(z) = \sum_{x = (x_1, \dots, x_n) \in \Z^n} q^{Q_{\Phi, \mathcal{B}}(x)} = \sum_{x = (x_1, \dots, x_n) \in \Z^n} e^{2 \pi i Q_{\Phi, \mathcal{B}}(x) z}.
\end{align*}

In order to discuss the automorphy of the theta series (see Theorem \ref{Automorphy} below), we need to introduce the notion of the level of a quadratic form.

To define the level of a quadratic form we first define the level of an even symmetric matrix. Thus, let $A \in M_{n \times n} (\Z)$ be an even symmetric matrix (recall that even means that its entries are integers and that its diagonal entries lie in $2\Z$). Then, as is well known, we have the relation
\begin{align*}
A \adj{(A)} = \det{(A)} I_n,
\end{align*}
where $\adj{(A)} \in M_{n \times n}(\Z)$ is the adjugate matrix of $A$. If $A$ is invertible, this relation implies that the product $\det(A) A^{-1} \in M_{n \times n}(\Z)$ is also an even symmetric matrix (see e.g. the first lemma in Chapter VI of Ogg's book \cite{Ogg69} and the subsequent discussion). Thus since this implies that $|\det{(A)}| A^{-1}$ is an even symmetric matrix in $M_{n \times n}(\Z)$, we define the level of $A$ to be the least positive integer $N$ such that $N A^{-1}$ is an even symmetric matrix in $M_{n \times n}(\Z)$. Moreover, it can be seen that if $P \in \op{GL}_{n}(\Z)$ is an invertible matrix, then $P^t A P$ also has level $N$. This allows us to define unambiguously the level of a quadratic form as follows.

\begin{definition}
The level of a quadratic form $\Phi: \Lambda \longrightarrow \Z$ is defined to be the level of any associated matrix $A_{\Phi, \mathcal{B}}$ corresponding to a $\Z$-basis $\mathcal{B}$ of $\Lambda$.
\end{definition}

Now we recall the definition of the space of modular forms that we will consider, following \cite{Koblitz93} and \cite{lehman}. First, we will require the quadratic residue symbol $\left(\dfrac{a}{b} \right)$, defined for $a, b \in \Z$ by the following properties.
\begin{enumerate}
\item If $p$ is an odd prime, then $\left( \dfrac{a}{p} \right)$ is just the Legendre symbol.
\item If $a \in \Z$ is odd, then $\left(\dfrac{a}{2} \right) = (-1)^{(a^2 - 1)/8}$.
\item We have $\left(\dfrac{a}{-1} \right) = 1$ if $a \geq 0$ and $\left(\dfrac{a}{-1} \right) = -1$ if $a < 0$.
\item $\left(\dfrac{a}{b} \right) = 0$ if $\gcd{(a, b)} > 1$.
\item $\left(\dfrac{1}{0} \right) = 1$ and $\left(\dfrac{a}{0} \right) = 0$ for $a \neq 1$.
\item $\left(\dfrac{a}{bc} \right) = \left(\dfrac{a}{b} \right) \cdot \left(\dfrac{a}{c} \right)$ for every $a, b, c \in \Z$.
\end{enumerate}

Second, if $d \in \Z \smallsetminus \{ 0\}$, there are integers $d_f$ and $d_s$ such that we can write $d$ uniquely as $d = d_{f} d_s^2$ with $d_f$ square-free. Then we put 
\begin{align*}
D_d := 
\begin{cases}
d_f & \text{if $d_f \equiv 1 \pmod 4$}\\
4d_f & \text{if $d_f \equiv 2, 3 \pmod{4}$}.
\end{cases}
\end{align*}
Note in particular that if $d$ is not a perfect square, then $D_d$ is just the discriminant of the quadratic number field $\Q(\sqrt{d})$. With this notation in place, we define the quadratic Dirichlet character $\chi_d(n):= \left(  \dfrac{D_d}{n} \right)$. This character has conductor $|D_d|$. 

\begin{definition}
Let $n \in \Z$ and $N \in \Z_{\geq 1}$. Moreover, assume that $4|N$ if $n$ is odd. Let $\chi$ be a Dirichlet character modulo $N$. Then a holomorphic function $f: \H \longrightarrow \C$ is called a modular form of weight $n/2$, level $N$ and character $\chi$, if it is holomorphic at every cusp of $\Gamma_0(N)$ and if for every $\gamma = \begin{pmatrix} a & b \\ c & d \end{pmatrix} \in \Gamma_0(N)$ it transforms as
\begin{align*}
f(\gamma \cdot z) = 
\begin{cases}
\chi(d) (cz + d)^{n/2} f(z) & \text{if $n$ is even},\\
\chi(d) \chi_c(d)^n \varepsilon_d^{-n} (cz + d)^{n/2} f(z) & \text{if $n$ is odd,}
\end{cases}
\end{align*}
where
\begin{align*}
\varepsilon_d := 
\begin{cases}
1 & \text{if $d \equiv 1 \pmod{4}$,}\\
i & \text{if $d \equiv 3 \pmod{4}$.}
\end{cases}
\end{align*}
The $\C$-vector space of all such forms is denoted by $M_{n/2}(\Gamma_0(N), \chi)$.
\end{definition}

The following result is a special case of a more general result of Shimura \cite[Proposition 2.1]{Shimura73} and it is stated in an equivalent way in \cite[p. 400]{lehman}.

\begin{theorem}\label{Automorphy}
Let $\Phi: \Lambda \longrightarrow \Z$ be a quadratic form on a free $\Z$-module $\Lambda$ of rank $n$. Let $N_{\Phi}$ be the level of the quadratic form $\Phi$ and let $A = A_{\Phi, \mathcal{B}}$ be the matrix of $\Phi$ with respect to some $\Z$-basis $\mathcal{B}$ of $\Lambda$. Define the integer $d_{\Phi}$ by
\begin{align*}
d_{\Phi} := 
\begin{cases}
\det{(A)} &\text{if $n \equiv 0 \pmod{4}$}\\
-\det{(A)} &\text{if $n \equiv 2 \pmod{4}$}\\
\dfrac{\det{(A)}}{2} &\text{if $n \equiv 1, 3 \pmod{4}$}.
\end{cases}
\end{align*} 
Then the theta function 
\begin{align*}
\theta_{\Phi}(z) := \sum_{\lambda \in \Lambda} q^{\Phi(\lambda)} = \sum_{\lambda \in \Lambda} e^{2 \pi i \Phi(\lambda) z}
\end{align*}
is a modular form in the space $M_{\frac{n}{2}}(\Gamma_0(N_{\Phi}), \chi_{d_{\Phi}})$. 
\end{theorem}

\section{Trace forms on totally real number fields}\label{LaTraza}

Let $K$ be a totally number field of degree $n$ and let $\mathcal{O}_{K}$ be its ring of integers. Then the trace zero submodule of $\mathcal{O}_K$, defined by
\begin{align*}
\Lambda_K = \mathcal{O}_K^{0} := \ker{(\Tr_{K/\Q})} = \{ \alpha \in \mathcal{O}_K \suchthat \Tr_{K/\Q}(\alpha) = 0 \},
\end{align*}
is a free $\Z$-module of rank $n - 1$. This is because since $\Tr_{K/\Q}: \mathcal{O}_K \longrightarrow \Z$ is a non-zero homomorphism and hence its image has rank $1$.

The following proposition gives us a quadratic form constructed from the trace. We include a proof for the convenience of the reader.

\begin{proposition}\label{ElRango}
The function $\Phi_K: \Lambda_K \longrightarrow \Z$ given by $\Phi_K(\alpha) := \frac{1}{2}\Tr_{K/\Q}(\alpha^2)$ is a positive definite quadratic form of rank $n - 1$ and its associated bilinear function $\phi_K: \Lambda_K \times \Lambda_K \longrightarrow \Z$ satisfies $\phi_K(\alpha, \beta) = \Tr_{K/\Q}(\alpha \beta)$.
\end{proposition}

\begin{proof}
First we need to show that $\Phi_K$ takes values in $\Z$. Let $\alpha \in \Lambda_K$.  Since $\Phi_K(\alpha)$ is rational it is enough to show that  $\Phi_K(\alpha)$ is an algebraic integer. Now, let the embeddings $K \hookrightarrow \R$ be $\sigma_1, \dots, \sigma_n$. Then for any $\alpha \in \Lambda_K$ we have
\begin{align*}
\Phi_K(\alpha) &= \frac{1}{2} \Tr_{K/\Q}(\alpha^2) = \frac{1}{2} \sum_{i = 1}^n \sigma_i(\alpha^2) = \frac{1}{2} \sum_{i = 1}^n \sigma_i(\alpha)^2\\
&= \frac{1}{2} \left( \left( \sum_{i = 1}^n \sigma_i(\alpha) \right)^2 - 2 \sum_{i < j} \sigma_i(\alpha) \sigma_j(\alpha) \right) \\
&= \frac{1}{2} \left( \Tr_{K/\Q}(\alpha)^2 - 2 \sum_{i < j} \sigma_i(\alpha) \sigma_j(\alpha) \right) \\
&= - \sum_{i < j} \sigma_i(\alpha) \sigma_j(\alpha),
\end{align*}
where we used that $\Tr_{K/\Q}(\alpha) = 0$. Since $\alpha$ is an algebraic integer, so are its conjugates $\sigma_{k}(\alpha)$ for all $k$, and in particular that implies that $\Phi_K(\alpha) = - \sum_{i < j} \sigma_i(\alpha) \sigma_j(\alpha)$ is an algebraic integer as we wanted to show. 

Now, to see that $\Phi_K$ is a quadratic form, observe that the linearity of the trace implies that for any $m \in \Z$ and $\alpha \in \Lambda_K$ we have
\begin{align*}
\Phi_K(m \alpha) = \frac{1}{2} \Tr_{K/\Q}(m^2 \alpha^2) = \frac{m^2}{2} \Tr_{K/\Q}(\alpha^2) = m^2 \Phi_K(\alpha).
\end{align*}
Also, the associated function $\phi_K: \Lambda_K \times \Lambda_K \longrightarrow \Z$, satisfies
\begin{align*}
\phi_K(\alpha, \beta) &= \Phi_K(\alpha + \beta) - \Phi(\alpha) - \Phi(\beta) = \frac{1}{2} \left( \Tr_{K/\Q}((\alpha + \beta)^2) - \Tr_{K/\Q}(\alpha^2) - \Tr_{K/\Q}(\beta^2) \right) \\
&= \frac{1}{2} \left( \Tr_{K/\Q}(\alpha^2 + 2\alpha \beta + \beta^2) - \Tr_{K/\Q}(\alpha^2) - \Tr_{K/\Q}(\beta^2) \right) = \Tr_{K/\Q}(\alpha \beta).
\end{align*}
Thus, since $\phi_K(\alpha, \beta) = \Tr_{K/\Q}(\alpha \beta)$, an easy calculation shows that this map is bilinear. Since $K$ is totally real the values of  $\Phi_K$  are sum of squares of real numbers, hence $\phi_K$ is a positive definite bilinear pairing. 
\end{proof}

\begin{lemma}\label{Discriminant}

Let $K$ be a degree $n$ number field of discriminant $d$. Suppose that $m$ is a positive integer such that $\Tr_{K/\Q}(O_{K})=m\Z$. Then, the discriminant of the quadratic form  $\Phi_K$ is equal to $\displaystyle (-1)^\frac{(n-1)(n-2)}{2} \frac{dn}{m^{2}}.$ In particular, if d is not divisible by the $n$th power of any integer bigger than $1$ then the discriminant of $\Phi_K$ is equal to $\displaystyle (-1)^\frac{(n-1)(n-2)}{2} dn.$

\end{lemma}

\begin{proof}

Since $O_{K}^{0}$ has $\Z$-rank equal to $n-1$ it is enough to show that  the determinant $d_0$ of a Gram matrix of the bilinear pairing  $\Tr_{K/\Q}(\cdot, \cdot)$ in any basis of $O_{K}^{0}$ is equal to $\frac{dn}{m^{2}}$. Since $\Z$ and $O_{K}^{0}$ are orthogonal with respect the trace pairing, the determinant of the trace pairing over  $\Z +O_{K}^{0}$ is equal to $nd_{0}$. On the other hand, since  $\Z +O_{K}^{0}$ is a subgroup of $O_{K}$ of full rank, the determinant is equal to $[O_{K} : \Z +O_{K}^{0} ]^{2}d $ and thus \[d_{0}= [O_{K} : \Z +O_{K}^{0} ]^{2}d/n.\] Since $\Z +O_{K}^{0}$ is a subgroup of $O_{K}$ that contains the Kernel of the trace map the group $O_{K}/(\Z +O_{K}^{0})$ is isomorphic to $m\Z/n\Z$, in particular $[O_{K} : \Z +O_{K}^{0} ] =\frac{n}{m}$ from which $d_{0}=\frac{nd}{m^{2}}$ follows. The final remark is due to the fact that $m^{n}$ divides $d$.

\end{proof}

Since the level of an integral quadratic form of discriminant $\Delta$ divides $2\Delta$ we have that: 

\begin{corollary}\label{TheLevel}

Let $K$ be a degree $n$ number field of discriminant $d$. Then, the level of the quadratic form $\Phi_K$ divides $2nd$.

\end{corollary}

\begin{remark}

In the case of $n=3$  and $d$ is fundamental it's not hard to observe, see  \cite{GMS1},  that the level is actually $nd$. This is not necessary the case in higher degrees, hence the best we can say is that the level divides $2nd$

\end{remark}

\subsubsection*{Proof of Theorem \ref{LaDefi}}
Recall that  $K$ is a totally real number field of degree $n$ with discriminant $d_K$. Moreover, $d_K$ is a fundamental discriminant and $\gcd{(n, d_K)} = 1$.  
\begin{proof}
It follows from Lemma \ref{Discriminant} that the form $\Phi_K$ has determinant $nd_{K}$, and by Corollary \ref{TheLevel}, its level is a divisor of $2nd_{K}$. By definition, the theta series associated to the form $\Phi_K$ is given by 

\begin{align*}
\theta_K(z) := \sum_{\alpha \in \mathcal{O}_K^{0}} e^{\pi i \Tr_{K/\Q}{(\alpha^2)} z}
\end{align*}

Since $\Phi_K$ has rank $n-1$, thanks to Proposition \ref{ElRango}, determinant $nd_{K}$ and level dividing $2nd_{K}$ it follows from Theorem \ref{Automorphy} that \begin{align*}
\theta_K \in M_{\frac{n - 1}{2}} \left( \Gamma_0(2nd_K), \left( \frac{\delta_{n} nd_{K}}{\cdot} \right) \right),
\end{align*}
where $\delta_n =\begin{cases}
 
(-1)^{\frac{n-1}{2}} &\text{if $n$ is odd,}\\
\frac{1}{2} &\text{if $n$ is even.}
\end{cases}$ \end{proof}

\section{Proof of Theorem \ref{MainTheorem}}

As was shown in section \ref{LaTraza}, the quadratic form $\Phi_K = \frac{1}{2} \Tr_{K/\Q}$ on the trace zero module $\Lambda_K = \mathcal{O}_K^0$ has rank $n - 1$. Since for any isomorphism  $\sigma : K \to L$ between number fields, and for all $\alpha \in K$, we have that $\Tr_{K/\Q}(\alpha) =\Tr_{L/\Q}(\sigma(\alpha))$ the implication $K \cong L \implies \theta_K = \theta_L$ is trivial. We prove the other implication analyzing the cases $n =2, 3, 4$ individually.

\subsection{n = 1} There is nothing to prove here.

\subsection{n = 2}
If $K$ is the real quadratic field  $\Q(\sqrt{d})$,  where $d >1$ is a square free integer, then a calculation shows that $\mathcal{O}_{K}^{0}=\Z\cdot \sqrt{d}$ and therefore $\displaystyle \theta_{K}(z)= \vartheta_{d}(z)= 1+ 2\sum_{n=1}^{\infty}e^{\pi i dn^2 z}.$ Taking $z = i$, the result follows from the injectivity of the function $\displaystyle f(x)=\sum_{n=1}^{\infty}e^{-\pi xn^2 }$ for $x > 0$.
 
\subsection{n = 3}
This was proved by the second author in \cite[Theorem 3.7]{GMS1}.
 
\subsection{n = 4}
Given a positive definite quadratic form $Q(x_1, \dots, x_n) \in \R[x_1, \dots, x_n]$ and a positive real number $t \in \R_{>0}$, we define the representation number $A(Q, t)$ by
\begin{align*}
A(Q, t):= \# \{ x \in \Z^n \suchthat Q(x) = t \}.
\end{align*}
Now, for the quadratic forms $\Phi_K$ and $\Phi_L$, choose bases $\mathcal{B}_K$ and $\mathcal{B}_L$ and let $A_{\Phi_K, \mathcal{B}_K}$ and $A_{\Phi_L, \mathcal{B}_L}$ be the corresponding matrices in $M_{3 \times 3}(\Z)$. Moreover, let $Q_{K, \mathcal{B}_K}$ and $Q_{L, \mathcal{B}_L}$ be the corresponding quadratic polynomials as in Section~\ref{QuadraticBackground}. Then, we have
\begin{align*}
\theta_K(z) = \sum_{\alpha \in \mathcal{O}_K^0} q^{\Phi_K(\alpha)} = \sum_{x \in \Z^3} q^{Q_{K, \mathcal{B}_K}(x)} = 1 + \sum_{n = 1}^{\infty} A(Q_{K, \mathcal{B}_K}, n) q^{n}
\end{align*}
and similarly for $\theta_L$. Therefore, since $\theta_K = \theta_L$, the corresponding representation numbers are equal, i.e., we have
\begin{align*}
A(Q_{K, \mathcal{B}_K}, n) = A(Q_{L, \mathcal{B}_L}, n)
\end{align*}
for every $n \geq 1$. Then since the forms $Q_{K, \mathcal{B}_K}$ and $Q_{L, \mathcal{B}_L}$ are positive definite and ternary, a result of Schiemann \cite[Theorem 4.4]{schiemann} implies that there is a matrix $U \in \op{GL}_3(\Z)$ such that 
\begin{align}\label{QuadraticRelation}
Q_{L, \mathcal{B}_L}(Ux) = Q_{K, \mathcal{B}_K}(x). 
\end{align}
Now, let the basis $\mathcal{B}_L = \{ \beta_1, \beta_2, \beta_3 \}$. Thus we define the isomorphism of $\Z$-modules $T: \mathcal{O}_K^0 \longrightarrow \mathcal{O}_L^0$ by $T(\alpha) := [\beta_1, \beta_2, \beta_3] U[\alpha]_{\mathcal{B}_K}$, i.e., the coordinates of $T(\alpha)$ in the basis $\mathcal{B}_L$ are given by $[T(\alpha)]_{\mathcal{B}_L} = U [\alpha]_{\mathcal{B}_K}$ for every $\alpha \in \mathcal{O}_K^0$. Then, by equations (\ref{QuadraticRelation}) and (\ref{Phi_Quadratic}) we see that $\Phi_L(T(\alpha)) = \Phi_K(\alpha)$ for every $\alpha \in \mathcal{O}_K^0$ and hence the quadratic forms $\Phi_K$ and $\Phi_L$ are isometric. Thus we have an isomorphism of quadratic spaces
\begin{align*}
(\mathcal{O}_K^0, \Phi_K) \cong (\mathcal{O}_L^0, \Phi_L).
\end{align*}
Then a recent theorem of the second author and Rivera-Guaca \cite[Theorem 2.12]{GMS3} implies that $K \cong L$, which completes the proof of the theorem.

\section{Computations and heuristics for linear independence}\label{Heuristics}

Suppose we fix the discriminant $\disc(K)$ of the field $K$ and ask a finer question about the theta series associated to the number fields of discriminant $\disc(K)$.  In particular, it was proved in \cite{GMS1} that for cubic number fields these forms were linearly independent.  We ask the same question here.


\subsection{Computational evidence}  The evidence we present below was computed in Sage \cite{Sage}.  The data related to the number fields was downloaded from the LMFDB \cite{lmfdb} and from the database described in \cite{malle} and due Kl\"uners and Malle.  The LMFDB contains totally real quartic number fields up to discriminant $10^7$ and the Kl\"uners--Malle database contains totally real quartic number fields up to discriminant $10^{10}$.  We used both data sets to compare and verify results.

We restrict to those fields that satisfy our conditions (totally real, Galois group equal to $S_4$ \cite{Kondo}).  We further restrict to those discriminants for which there is more than one field of that discriminant.  This leave us with $1301494$ fields to consider.  It was shown computationally that for all the discriminants, the resulting theta series were linearly independent.  The code and examples can be found at \cite{code}.

\subsection{Two theta series represent the same prime}\label{sec:sameprime}

In \cite{GMS1} it was shown that if two integral binary quadratic forms of the same discriminant both represented the same prime $p$, then the two forms are equivalent.  This was a key step in proving the linear independence of the set of quadratic forms associated to number fields of a fixed discriminant.  We attempted to follow the same approach.  But, as shown in Table~\ref{tbl:same-prime} there are three quartic fields of discriminant $35537$ satisfying our conditions but two of them represent the same prime.

\begin{table}[H]
\begin{center}
\begin{tabular}{lcc}
Polynomial $f$ & Quadratic form coefficients & Theta series $\theta_K$ \\\hline
 $x^4 - 2x^3 - 9x^2 + 5x + 16$ & $\begin{pmatrix}95506 & 93618 & 261632 \\  * & 22943 & 128229\\  * &  * & 179181  \end{pmatrix}$ & $1 + 2q^{23} + 2q^{27} + O(q^{30})$\\\hline
 $x^4 - x^3 - 8x^2 - 3x + 4$ & $\begin{pmatrix}321 & 1038 & -505 \\  * & 851 & -861 \\ * &  * & 245 \end{pmatrix}$& $1 + 2q^{21} + 4q^{23} + O(q^{30})$\\\hline
 $x^4 - 2x^3 - 5x^2 + 5x + 4$ & $\begin{pmatrix}527 & 3957 & 7078 \\  * & 7439 & 26613 \\  * &  * & 23805\end{pmatrix}$ & $ 1 + 2q^{11} + 2q^{26} + O(q^{30})$
\end{tabular}
\end{center}
\caption{The field is $K = \Q[x]/\langle f(x) \rangle$, where $f(x)$ is the polynomial given in the first column.  The symmetric matrix $A = (a_{ij})$ that is displayed in the second column contains the coefficients of $\displaystyle{\Phi_K(x_1, x_2, x_3) = \sum_{i = 1}^3 a_{ii}x_i^2 + 2\sum_{j > i} a_{ij}x_i x_j}$.  Finally, the third column displays the first terms of the $q$-expansion of the theta series $\theta_K$.}\label{tbl:same-prime}
\end{table}

By computing the $\theta$-series of the Quadratic forms associated to $K_1$ and $K_2$ in Table~\ref{tbl:same-prime} we see they both represent $23$, rendering the approach in \cite{GMS1} not viable.  This led us to consider a second approach we describe next.

\subsection{Two theta series with the same minimum}

If we knew that a set of quadratic forms of the same discriminant had different positive minima, then it would be easy to show that the set of forms were linearly independent.  After the approach in \cite{GMS1} described in Section~\ref{sec:sameprime} was seen not to be viable, we considered this approach.  As the following example shows, this approach was not viable either.

We remark that the notation is the same as in Table~\ref{tbl:same-prime}.  We consider the two quartic number non-isomorphic number fields with discriminant $4024049$.  Let $K_1$ have polynomial $x^4 - x^3 - 37x^2 + 46x + 20$ and $K_2$ have polynomial $x^4 - 2x^3 - 41x^2 - 49x + 95$.  The quadratic form associated to $K_1$ has coefficients
\[
\begin{pmatrix}10684425 & 71591860 & 106295143 \\ * & 119926773 & 356119635 \\ * &  * & 264372149 \end{pmatrix}
\]
and the quadratic form associated to $K_2$ has coefficients
\[
\begin{pmatrix}153730865 & 332210617 & 357281618 \\ * & 179475823 & 386040717 \\  * &  * & 207587063 \end{pmatrix}.
\]
If we proceed to compute the $\theta$-series of the two quadratic forms we see that they both represent 43 and 43 is the smallest positive integer they both represent.

We had hoped that we could attempt to prove that theta series corresponding to fields of the same discriminant were linearly independent by showing that such forms had different positive minima:
\begin{align*}
\theta_{K_1}(q) &= 1 + 2q^{43} + 2q^{172} + O(q^{200})\\
\theta_{K_2}(q) &= 1 + 2q^{43} + 2q^{170} + 2q^{172} + 2q^{187} + O(q^{200})
\end{align*}
As this examples shows, this approach cannot yield a proof.

\subsection{Heuristic evidence}

In order to have a chance for the set of $\theta$-series to be linearly, independent, we need that the number of quartic fields $K$ of discriminant $d_K$ that satisfy our conditions be less than the dimension of the corresponding space of modular forms of weight $\tfrac32$. To that end we present an informal argument for why there are few $\theta$-series of quartic number fields relative to the dimension of the corresponding space of weight $\tfrac32$ modular forms.

\subsubsection*{Bounds for dimensions of spaces of weight $\tfrac32$ modular forms}  We want to calculate
\[
\dim M_{\tfrac32}(\Gamma_0(8d_K),\kro{2d_K}{\cdot}).
\]
There are formulas in \cite{cohen-oesterle} that give the dimension the space of forms that are genuinely of level $8d_K$; we denote this space by $\dim M_{\tfrac32}^{\textrm{new}}(\Gamma_0(8d_K),\kro{2d_K}{\cdot})$.  It suffices for our purposes to make the following crude estimates:
\begin{align*}
\dim M_{\tfrac32}(\Gamma_0(8d_K),\kro{2d_K}{\cdot}) &> \dim M_{\tfrac32}^{\textrm{new}}(\Gamma_0(8d_K),\kro{2d_K}{\cdot})\\
&> \dim S_{\tfrac32}^{\textrm{new}}(\Gamma_0(8d_K),\kro{2d_K}{\cdot}\\
& > \dim S_{\tfrac32}^{+,\textrm{new}}(\Gamma_0(8d_K),\kro{2d_K}{\cdot})\\
& = \dim S_2^{\textrm{new}}(\Gamma_0(2d_K),\kro{2d_K}{\cdot})
\end{align*}
Also in \cite{cohen-oesterle} one can find formulas for $\dim S_2^{\textrm{new}}(\Gamma_0(2d_K),\kro{2d_K}{\cdot})$ and as these suffice and are simpler, we use them to get a lower bound on the dimension of the space we are interested in.

In particular, letting $N=2d_{K}$, for $k=2$ we get
\begin{multline*}
\dim S_2^\textrm{new}(\Gamma_0(2d_K),\kro{2d_K}{\cdot}) - \dim M_0^\textrm{new}(\Gamma_0(2d_K),\kro{2d_K}{\cdot})=\\
\frac{N}{12}\prod_{p\mid N} (1+\tfrac1p) - \tfrac12 \prod_{p\mid N} \lambda(r_p,s_p,p)+ \varepsilon_k\sum_{\substack{x\pmod{N}\\x^2+1\equiv 0 \pmod{N}}} \kro{2d_K}{x} +\\ \mu_k \sum_{\substack{x\pmod{N}\\x^2+x+1\equiv 0 \pmod{N}}} \kro{2d_K}{x}.
\end{multline*}
In our case, since $k=2$, we know that $\mu_k=0$ and we also know that $\dim M_0^\textrm{new}(\Gamma_0(2d_K),\kro{2d_K}{\cdot})=1$.  We consider each summand on its own.

If $p\mid 2d_K$ then we define $r_p$ to be the exponent of $p$ in the prime factorization of $2d_k$ and we define $s_p$ to be the exponent of $p$ in the prime factorization of the conductor of $\kro{2d_K}{\cdot}$ (since the character is quadratic, we know that its conductor is $2d_K$.  Recall that since $d_K$ is odd, we know $2d_K$ is square-free and observe that $r_p=s_p=1$.  Now
\[
\lambda(r_p,s_p,p)=\begin{cases} p^{r^\prime}+p^{r^\prime-1} & \text{ if } 2s_p \leq r_p = 2 r^\prime\\
2p^{r^\prime} & \text{ if } 2s_p \leq r_p = 2r^\prime + 1\\
2p^{r_p-s_p} & \text{ if } 2s_p > r_p.
\end{cases}
\]
We are in the third case and so the second term becomes $\tfrac12 \prod_{p\mid 2d_K} 2$ which is clearly less than $\sqrt{2d_K}$.  Now we consider the next term.  The worst case is that all the primes that divide $d_K$ are equivalent to 1 modulo 4.  Since we want a crude upper bound we assume that $\kro{2d_K}{x} =1$ and so we see that the third term is also less than $\sqrt{2d_K}$.  Thus we conclude that $\dim M_{\tfrac32}(\Gamma_0(8d_K),\kro{2d_K}{\cdot}) > O(d_{K})$. As we see next, this is far greater than the number of $\theta$-series we construct. 

\subsubsection*{The number of totally real quartic number fields}

Let $d$ be a positive square free discriminant and let $\mathcal{Q}_{d}$  be the set of isomorphism classes of totally real quartic fields of discriminant $d$. Let $\mathcal{F}_{d}:= \{ \theta_K: K \in \mathcal{Q}_{d} \}$.   Now we estimate an upper bound of order $d^{0.62}$ for the number of elements in $\mathcal{F}_{d}$. In particular,  $\#\mathcal{F}_{d}$  is much smaller than $O(d)$ which is a lower bound for the dimension of the space of modular forms in which the set $\mathcal{F}_{d}$ lives. Based on this, the computational evidence presented at the beginning of this section, and the analogous result for cubic fields (see \cite{GMS1} ) we expect that the set  $\mathcal{F}_{d}$ is linearly independent. More precisely:

\begin{conjecture}
Let $d$ be a positive square free discriminant and let $\mathcal{Q}_{d}$  be the set of isomorphism classes of totally real quartic fields of discriminant $d$. Let $\mathcal{F}_{d}:= \{ \theta_K: K \in \mathcal{Q}_{d} \}$ where $\theta_{K}$ is the form defined in Theorem \ref{LaDefi}. Then, the set $\mathcal{F}_{d}$ is lineally independent.
\end{conjecture}

For quartic fields of fundamental discriminant there is a classical duality between quartic fields and $2$-torsion elements of their cubic resolvents. More precisely, if $K$ is the cubic resolvent field of $L$ then the field $L$ corresponds to an index $2$ subgroup of $Cl(K)$. In a similar fashion, the field $K$ corresponds to an index $3$ subgroup of its quadratic resolvent. Here we recall briefly how such correspondences occur; see \cite{BhargavaQuartics} and \cite[\S 3.1]{bhargava}. Let $L$ be a quartic field of fundamental discriminant $d$. By a result of \cite{Kondo} if $\widetilde{L}$ is the Galois closure of $L$ then 
$S_{4} \cong \gal(\widetilde{L}/\Q)$. Hence, up to conjugation, $\widetilde{L}$ contains a unique cubic field $K$, called the resolvent of $L$. The Galois group  $\gal(\widetilde{L}/\Q)$ is isomorphic to the dihedral group of order $8$; let $E$ be the quadratic extension of $K$ that corresponds, via Galois, to the unique cyclic subgroup of order $4$. As it turns out, the extension $E/K$ is unramified and thus, by class field theory, this corresponds to a index $2$ subgroup of $Cl(K)$. Since $d$ is square free this is also the discriminant of $K$, hence if  $\widetilde{K}$ denotes the Galois closure of $K$ then $S_{3} \cong \gal(\widetilde{K}/\Q)$. Furthermore, the unique quadratic field inside $\widetilde{K}$, the quadratic resolvent of $K$ is the quadratic field of discriminant $d$ i.e., $\Q(\sqrt{d})$. The Galois extension  $\widetilde{K}/\Q(\sqrt{d})$ is unramified hence it  corresponds to an index $3$ subgroup of $Cl(\Q(\sqrt{d}))$.

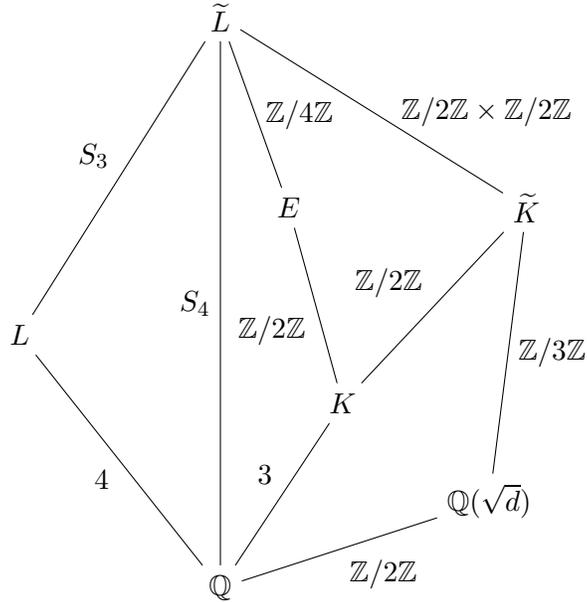
\[
\begin{tikzpicture}[scale=1.3]
  \node at (3.25,0) (1) {$\mathbb{Q}$};
  \node at (1.2,2.59) (2) {$L$};
  \node at (4.5,1.89) (3) {$K$};
  \node at (6,0.89) (4) {$\mathbb{Q}(\sqrt{d})$};
  \node at (3.95,3.89) (5) {$E$};
  \node at (6.38,3.89) (6) {$\widetilde{K}$};
  \node at (3.25,5.84) (7) {$\widetilde{L}$};
  \draw[-] (1)edge node[auto] {$S_{4}$} (7) (1) edge node[auto] {$4$}(2) (1) edge node[auto] {$3$}(3) (1) edge node[below right] {$\Z/2\Z$} (4) 
     (2)edge node[auto] {$S_{3}$} (7) (3) edge node[auto] {$\Z/2\Z$} (5) (3) edge node[auto] {$\Z/2\Z$} (6)  (4) edge node[right] {$\Z/3\Z$} (6) (5) edge node[right] {$\Z/4\Z$} (7) (6)edge node[right] {\ \ $\Z/2\Z \times \Z/2\Z $} (7) ;
 
\end{tikzpicture}
\]

The above analysis is useful in bounding the number of quartic fields of fundamental discriminant equal to $d$. Given a finite abelian group $G$ and a prime $p$ it is not hard to see that the number of index $p$ subgroups of $G$ is given by \[\frac{p^{{\rm rank}_{p}(G)}-1}{p-1}\] where ${\rm rank_{p}}(G)=\dim_{\F_{p}}(G \otimes \F_{p}) $

\begin{proposition}

Let $d$ be a fundamental discriminant and let $n_{4}(d)$ be the number of isomorphism classes of quartic fields with discriminant $d$. Then, $n_{4}(d) \ll O(d^{0.62}).$

\end{proposition}

\begin{proof}

Let $T_{d}$ be the number of cubic fields, up to isomorphism,  of discriminant equal to $d$ and let $\{K_{1},...,K_{T_{d}}\}$ be a set of representatives of such fields. The above analysis yields 
\[
n_{4}(d)=\sum_{i=1}^{T_{d}} \left(2^{{\rm rank}_{2}(Cl(K_{i}))}-1\right) < \sum_{i=1}^{T_{d}} \left(2^{{\rm rank}_{2}(Cl(K_{i}))}\right) = \sum_{i=1}^{T_{d}} Cl(K_{i})[2] \ll d^{0.28}T_{d}
\] 
where the last inequality is due to \cite{BhargavaFrankAndAll}. By the analysis on cubics above \[T_{d}=\frac{3^{{\rm rank}_{3}(Cl(\Q(\sqrt{d})))}-1}{2} = (Cl(\Q(\sqrt{d})[3]-1)/2 \ll d^{1/3}\] where the last inequality is due to \cite{EllenVenka}. The result follows since $0.28 + 1/3 < 0.62.$

\end{proof}

As an immediate corollary of Theorem \ref{MainTheorem} we have:

\begin{corollary}

Keeping with the notation of this section we have that $\#\mathcal{F}_{d} \ll O(d^{0.62}).$
\end{corollary}

\subsection{Higher degree fields}

For all the totally real number fields of degrees 5, 6, 7 in the Kl\"uners--Malle database \cite{malle}, the set of associated theta series is linearly independent (in degree 5 there were 5870 fields where the number of associated theta series was greater than 1; in degree 6 there were 236 fields where the number of associated theta series was greater than 1; in degree 7 there were 16 fields where the associated number of associated theta series was greater than 1).

\bibliographystyle{plain}
\bibliography{oaxaca}

%
%
%
%
%
%
%
%
%
%
%

\end{document}